\documentclass{amsart}

\usepackage{amsmath, amscd, amssymb}
\usepackage[all]{xy}
\usepackage{color}

\newcommand{\cal}{\mathcal}

\newcommand{\bC}{{\Bbb C}}

\newcommand{\bP}{{\Bbb P}}
\newcommand{\bQ}{{\Bbb Q}}

\newcommand{\bZ}{{\Bbb Z}}

\newcommand{\cO}{{\cal O}}

\newcommand{\cZ}{\mathcal{Z}}

\newcommand{\ds}{\displaystyle}
\newcommand{\Int}{\displaystyle\int}

\newtheorem{theorem}{Theorem}[section]

\newtheorem{proposition}[theorem]{Proposition}
\newtheorem{lemma}[theorem]{Lemma}

\theoremstyle{remark}
\newtheorem{remark}{Remark}[section]

\theoremstyle{definition}

\newtheorem{conjecture}{Conjecture}

\definecolor{light-pink}{rgb}{1,.90,.90}
\definecolor{light-green}{rgb}{.95,1,.95}

\definecolor{yellow}{rgb}{1,1,0}
\definecolor{orange}{rgb}{1,.7,0}
\definecolor{red}{rgb}{1,0,0}
\definecolor{white}{rgb}{1,1,1}

\definecolor{A}{rgb}{.75,1,.75}

\begin{document}

\title{Tautological Integrals on Symmetric Products of Curves}

\author{Zhilan \uppercase{Wang}}             
\address{Institute of Mathematics, Academy of Mathematics and Systems Science, Chinese Academy of Sciences, Beijing 100190, P. R. China}
\email{zlwang@amss.ac.cn}
   

\begin{abstract}
We propose a conjecture on the generating series of Chern numbers of tautological bundles on symmetric products of curves and establish the rank 1 and rank -1 case of this conjecture. Thus we compute explicitly the generating series of integrals of Segre classes of tautological bundles of line bundles on curves, which has a similar structure as Lehn's conjecture for surfaces.
\end{abstract}      

\maketitle

\section{Introduction}

Let $X$ be a smooth quasi-projective connected complex variety of dimension $d$, and denote by $X^{[n]}$ the Hilbert scheme of $n$ points on $X$. Let $\cZ_n\subset X\times X^{[n]}$ be the universal family with natural projections $p_1: \cZ_n \to X$ and $\pi: \cZ_n \to X^{[n]}$ onto the $X$ and $X^{[n]}$ respectively.
For any locally free sheaf $F$ on $X$, let $F^{[n]} = \pi_*(\cO_{\cZ_n}
\otimes p_1^*F)$, which is called the tautological sheaf of $F$. 

When $d=2$, many invariants of the Hilbert schemes of points
on a projective surface can be determined explicitly by the corresponding invariants
of the surface, including the Betti numbers \cite{Got}, Hodge numbers \cite{Got-Soe},
cobordism classes \cite{Ell-Got-Leh}, and elliptic genus \cite{Bor-Lib}, etc.. G. Ellingsrud, L. Göttsche and M. Lehn showed in \cite{Ell-Got-Leh} that for a polynomial in Chern classes of tautological sheaves and the tangent bundle of $X^{[n]}$, there exists  a universal polynomials in Chern classes of the corresponding sheaves and the tangent bundle on $X$ such that the integrals of these two polynomials over $X^{[n]}$ and $X$ are equal. A direct consequence is that the generating series of certain tautological integrals can be written in universal forms of infinite products; though it is not easy to find explicit expressions. For example, various authors have considered the computation of the integrals of top Segre classes of tautological sheaves of a line bundle on a surface \cite{ellingsrud1996bott,le1987formules,  Leh, tikhomirov1994standard, tikhomirov1994top, troshina1994degree}. M. Lehn made a conjecture on the generating series as follows:

\begin{conjecture}\label{conj:Lehn}(M. Lehn \cite{Leh})
For a smooth projective surface S and a line bundle L on it, define $$N_n=\int_{S^{[n]}}s_{2n}(L^{[n]}),$$

then

\begin{equation}
\sum\limits_{n\geq0}N_n z^n=\frac{(1-k)^a(1-2k)^b}{(1-6k+6k^2)^c}.
\end{equation}

Here $a=HK-2K^2$, $b=(H-K)^2+3\chi(\mathcal{O}_S)$ and $c=\frac{1}{2}H(H-K)+\chi(\mathcal{O}_S)$, where $H$ is the corresponding divisor of $L$ and $K$ is the canonical divisor.

And $$k=z-9z^2+94z^3-\cdots\in \bQ[[z]]$$

is the inverse of the function
\begin{equation*}
z=\frac{k(1-k)(1-2k)^4}{(1-6k+6k^2)^3}.
\end{equation*}
\end{conjecture}

This conjecture is still open up to now. Recently M. Marian, D. Oprea and R. Pandharipande showed that this conjecture holds for K3 surfaces \cite{Mar-Opr-Pan} by considering integrals over Quot schemes and the recursive localization relations.

J. V. Rennemo \cite{rennemo2012universal} gave a generaliztion of the theorem of G. Ellingsrud, L. Göttsche and M. Lehn: when $d=1$ and $d=2$ the universal property of polynomials in Chern classes of tautological sheaves and tangent bundles holds; when $d>2$, one should consider the universal property of integrals of polynomials only in Chern classes of tautological sheaves over geometric subsets of $X^{[n]}$.

In this article, we focus on the case of $d=1$. For $C$ a smooth projective curve, the Hilbert scheme of $n$ points on $C$ is isomorphic to the $n$-th symmetric product, so it is a smooth projective variety of dimension $n$. We have the following conjecture:


\begin{conjecture}\label{conj:Cn_n}
For $C$ a smooth projective curve and $E_r$ a vector bundle of rank $r$ on $C$, one has

\begin{equation} \label{eqn:Cn_n}
\sum_{n=0}^{\infty} z^n \int_{C^{[n]}} c(E_r^{[n]})=
\exp ( \sum_{n=1}^{\infty} \frac{z^n}{n} (A_n^rd_r+B_n^r e)),
\end{equation}

\begin{equation} \label{eqn:Cn_nminus}
\sum_{n=0}^{\infty} \frac{z^n}{n} \int_{C^{[n]}} c(-E_r^{[n]})=
\exp (\sum_{n=1}^{\infty} \frac{z^n}{n} (C_n^r(-d_r)+D_n^r e)),
\end{equation}

Here $c(E_r^{[n]})$ is the total Chern class, $d_r=\Int_C c(E_r)$ and $e$ is the Euler number of $C$. $A_n^r$, $B_n^r$, $C_n^r$ and $D_n^r$ are integers depending only on $r$ and $n$, which satisfy

$A_n^r=(-1)^{n+1}\ds\binom{rn-1}{n-1}$, $C_n^r=(-1)^{n}\ds\binom{-rn-1}{n-1}=(-1)^{n-1}A_n^{r+1}$, $D_n^r=(-1)^nB_n^{r+1}$.

\end{conjecture}

Conjecture \ref{conj:Cn_n} shares some similarities as in the surface case, which has been established in an unpublished work by Jian Zhou and the author and will appear in a subsequent work.

We will see the existence of the universal coefficients $A_n^r$, $B_n^r$, $C_n^r$ and $D_n^r$ is a direct consequence of Theorem \ref{thm:Universal} in Section \ref{sec:Uni} and it will be explained in Section \ref{subset:proof}. The mysterious part of Conjecture \ref{conj:Cn_n} to the author is that $A_n^r$, $B_n^r$, $C_n^r$ and $D_n^r$ are all integers and there are relationships between them. More precisely, if two vector bundles $E_r$ and $E_{r+1}$ of ranks $r$ and $r+1$ respectively satisfy $\int_C c(E_r)=\int_C c(E_{1+r})$, then it is implied by Conjecture \ref{conj:Cn_n} that
\begin{equation*}
\sum_{n=0}^{\infty} z^n \int_{C^{[n]}} c(-E_r^{[n]})=\sum_{n=0}^{\infty} (-z)^n \int_{C^{[n]}} c(E_{1+r}^{[n]})
\end{equation*}

For some special $r$, we can determine $B_n^r$ explicitly and prove the conjecture.

When $r=1$, $B_n^1=0$. We have the following theorem:

\begin{theorem}\label{thm:Cn_1}
For $C$ a smooth projective curve and $L$ a line bundle on $C$, one has
\begin{eqnarray} \label{eqn:Cn_1}
&& \sum_{n=0}^{\infty} z^n \int_{C^{[n]}} c(L^{[n]}) =
\exp ( \sum_{n=1}^{\infty} \frac{(-1)^{n+1}}{n} z^n\int_{C} c(L)).
\end{eqnarray}
\end{theorem}

For the rank -1 case, it is the generating series of the integrals of top Segre classes of tautological sheaves of a line bundle on $C$. We can also prove the conjecture in this case. Analogous to Conjecture \ref{conj:Lehn}, we have the following theorem:

\begin{theorem}\label{thm:Seg1}
For $C$ a smooth projective curve and $L$ a line bundle on $C$, one has 
\begin{equation} \label{eqn:Seg1}
\begin{split}
& \sum_{n=0}^{\infty} z^n \int_{C^{[n]}}s(L^{[n]})\\
=&\exp(\sum_{n=1}^{\infty}\frac{z^n}{n}(-\binom{2n-1}{n-1}d+(4^{n-1}-\binom{2n-1}{n-1})e))\\
=&\frac{(1-k)^{e+d}}{(1-2k)^{\frac{e}{2}}}.
\end{split}
\end{equation}
Here $s(L^{[n]})$ is the total Segre class, $d$ is the degree of the line bundle $L$, $e$ is the Euler number of $S$ and $z=k(1-k)$.
\end{theorem}

Theorem \ref{thm:Seg1} is related to an enumerative problem. A. S. Tikhomirov in \cite{tikhomirov1994standard} has interpreted $N_n$ in Conjecture \ref{conj:Lehn} as the number of ($n-2$)-dimensional $n$-secant planes in the image to the surface in $\bP^{3n-1}$. In a similar fashion, $(-1)^n\ds\int_{C^{[n]}}s(L^{[n]})$ counts the number of $n$-secant ($n-2$)-planes to $C$ in $\bP^{2n- 2}$. For example, it is easy to compute from Theorem \ref{thm:Seg1} that $\ds\int_{C^{[2]}}s(L^{[2]})=\ds\frac{1}{2}(d^2-3d+2-2g)$, which coincides the classical formula of the number of nodes of a curve in $\bP^2$. 

In 2007 Le Barz \cite{le2007formule} and E. Cotterill \cite{cotterill2011geometry} have already independently derived the generating formula of such numbers. Le Barz's approach is via the multisecant loci, and E. Cotterill uses a formula by Macdonld (cf. \cite{arbarello2011geometry} Chapter VIII, Prop. 4.2) and the graph theory. However, our method is different from Le Barz's and E. Cotterill's.

We use the similar strategy used in \cite{wang2014tautological} to prove the above theorems. Firstly we establish a universal formula theorem for curves as Theorem 4.2 in \cite{Ell-Got-Leh} for surfaces, and hence we only need to prove the cases of certain line bundles on $\bP^1$. Using the natural torus action on $\bP^1$ and the induced action on the Hilbert scheme, we can consider the equivariant case of $\bP^1$, or we can reduce it further to the equivariant case of $\bC$, which becomes a combinatoric problem.

\begin{remark}\label{rmk:proof}

After the first version of this article, Professor Oprea tells the author in an email that Conjecture \ref{conj:Cn_n} has been solved by M. Marian and himself, and the universal coefficients are also explicitly determined by them in \cite{Mar-Opr}. To be more precise, the relationships between the universal coefficients conjectured in Conjecture \ref{conj:Cn_n} do hold; if we write $$C(z)=\sum\limits_{n=1}^{\infty}\ds\frac{z^n}{n} C_n^r , D(z)=\sum\limits_{n=1}^{\infty}\ds\frac{z^n}{n} D_n^r,$$ then $$C(-t(1-t)^r)=-\log(1+t)$$ and $$D(-t(1+t)^r)=\ds\frac{r+1}{2}\log(1+t)-\ds\frac{1}{2}\log(1+t(r+1))$$.

A direct consequence of the above is that (\ref{eqn:Cn_nminus}) can be written in a form which is similar to Theorem \ref{thm:Seg1}:

\begin{equation*} 
\begin{split}
&\sum_{n=0}^{\infty} z^n \int_{C^{[n]}} c(-E_r^{[n]})\\=&
\frac{(1-k)^{\frac{r+1}{2}e+d}}{(1-(r+1)k)^{\frac{e}{2}}},
\end{split}
\end{equation*}

where $z=k(1-k)^r$.

\end{remark}

\section{Universal properties of tautological integrals over Hilbert schemes of points on curves}\label{sec:Uni}

Let $C$ be a smooth projective connected curve. It is well-known that $C^{[n]}$ is smooth of dimension $n$ and in particular isomorphic to the $n$-th symmetric product. In this section, we will see that the theorems on the cobordism rings of Hilbert schemes of points of surfaces established by Ellingsrud, G\"ottsche and Lehn in \cite{Ell-Got-Leh} can be generalized to curves.

We will follow what has been done in \cite{Ell-Got-Leh}. Let $\Omega=\Omega^U\otimes\bQ$ be the complex cobordism ring with rational coefficients. For a smooth projective curve $C$ we denote its cobodism class by $[C]$, and define an invertible element in the formal power series ring $\Omega[[z]]$:

$$H(C):=\sum\limits_{n=0}^{\infty}[C^{[n]}]z^n.$$

We have the following theorem:

\begin{theorem}\label{thm:GenEGL1}
$H(S)$ depends only on the cobordism class $[C]\in\Omega$. 
\end{theorem}

Two stably complex manifolds have the same cobordism class if and only if their collection of Chern numbers are identical. Theorem \ref{thm:GenEGL1} is proved in \cite{rennemo2012universal} and there is a generalized version of the this theorem:

For a smooth projective variety $X$, let $K(X)$ be the Grothendieck group generated by locally free sheaves. Let $E_1,\cdots,E_m\in K(C)$ and $r_1,\cdots,r_m$ are the ranks respectively.

\begin{theorem}\label{thm:GenEGL2}(J. V. Rennemo \cite{rennemo2012universal})
Let P be a polynomial in the Chern classes of $C^{[n]}$ and the Chern classes of $E_1^{[n]},\cdots,E_m^{[n]}$. Then there is a universal polynomial $\tilde{P}$, depending only on $P$, in the Chern classes of the tangent bundles of $C^{[n]}$, the ranks $r_1,\cdots,r_m$ and the Chern classes of $E_1,\cdots,E_m$, such that

\begin{equation*}
\int_{C^{[n]}}P=\int_S \tilde{P}.
\end{equation*}

\end{theorem}



These theorems can be used in the computations of generating series of tautological integrals. Let $\Psi:K(X)\to H^{\times}$ be  a group homomorphism
from the additive group $K(X)$ to  the multiplicative group $H^{\times}$ of
units of $H(X;\bQ)$.
We require $\Psi$ is functorial with
respect to pull-backs and is a polynomial in Chern classes of its
argument. Also let $\phi(x)\in\bQ[[x]]$ be a formal power series and
put $\Phi(X):=\phi(x_1)\cdots\phi(x_n)\in H^*(X;\bQ)$ with
$x_1,\cdots,x_n$ the Chern roots of $T_X$.
For $x\in K(X)$, define a power series in $\bQ[[z]]$ as follows:
\begin{eqnarray*}
H_{\Psi,\Phi}(X,x):=\sum\limits_{n=0}^{\infty}\int_{X^{[n]}}\Psi(x^{[n]})\Phi(X^{[n]})z^n.
\end{eqnarray*}

\begin{theorem} \label{thm:Universal}
For each integer $r$ there are universal power series
$A_i\in\bQ[[z]]$, $i=1, 2$, depending only on $\Psi$,
$\Phi$ and $r$, such that for each $x\in K(C)$ of rank $r$ we have

\begin{eqnarray*}
H_{\Psi,\Phi}(C,x)=\exp(\int_C(c_1(x)A_1+c_1(C)A_2)).
\end{eqnarray*}
\end{theorem}

The proof of the above theorem is similar as the proof of Theorem 4.2 in \cite{Ell-Got-Leh} so we omit the details here. The main idea is that $H_{\Psi,\Phi}$ factors through $\bQ^2$ to $\bQ[[z]]$ and we can choose ($\bP^1,r\mathcal{O}$) and ($\bP^1,(r-1)\mathcal{O}\oplus\mathcal{O}(-1)$) as the "basis".\\

\section{Proof of theorems}\label{sec:proof}

\subsection{Localizations on Hilbert schemes of points}\label{subsec:loc}

The linear coordinates on $\bC^{[n]}$ are given by
$p_i(z_1, \dots, z_n) = z_1^i + \cdots + z_n^i$.
The induced torus action on $\bC^{[n]}$ is given by
$$q \cdot p_i = q^ip_i, \;\;\;q=\exp(t) \in \bC^*.$$

This action has only one fixed point at $p_1 = \cdots = p_n = 0$,
and the tangent bundle and the tautological bundle $\mathcal{O}_{\bC}^{[n]}$ have the following weight decompositions at this point:
\begin{eqnarray*}
&&T_\bC^{[n]} = q^{-1} + \cdots + q^{-n},\\
&&\mathcal{O}_{\bC}^{[n]}=1+q+\cdots+q^{n-1}.
\end{eqnarray*}

For $A=(a_1,\cdots,a_r)$, where $a_1,\cdots,a_r\in\bZ$, denote by $\mathcal{E}_r^A=\mathcal{O}_{\bC}^{a_1}\oplus\cdots\oplus\mathcal{O}_{\bC}^{a_r}$ the rank $r$ T-equivariant vector bundle of weight $(a_1,\cdots,a_r)$. The tautological bundle $(\mathcal{E}_r^A)^{[n]}$ has the following weight decomposition at the fixed point:

\begin{eqnarray*}
(\mathcal{E}_r^A)^{[n]}= \sum\limits_{i=1}^{r}q^{a_i}(1+q+\cdots+q^{n-1}).
\end{eqnarray*}

For a smooth (quasi-)projective curve $C$ which admits a torus action with isolated fixed points $P_1,\cdots,P_l$ and $u_i=q^{c_i}$ the weights of $T^*_{P_i}C$, this torus action induces a $T$-action on $S^{[n]}$. The fixed points on $S^{[n]}$ are parameterized by nonnegative integers $(n_1,\cdots,n_l)$ such that $$n_1+\cdots+n_l=n.$$

The weight decomposition of the tangent space at the fixed point is given by:
\begin{eqnarray*}
\sum\limits_{i=1}^{l}(u_i^{-1}+\cdots+u_i^{-n_i}).
\end{eqnarray*}

Suppose $E$ is a rank $r$ equivariant vector bundle on $C$ such that 
\begin{eqnarray*}
E|_{P_i}=q^{a_1}+\cdots+q^{a_r}.
\end{eqnarray*}

Then the weight of $E^{[n]}$ at the fixed point $(n_1,\cdots,n_l)$ is given by
\begin{eqnarray*}
\sum\limits_{i=1}^{l}\sum\limits_{j=1}^{r}(t^{a_j}(1+u_i+\cdots+u_i^{n_i-1})).
\end{eqnarray*}

Let $\psi(x)\in\bQ[[x]]$ be a formal power series. For $x\in K(C)$ of rank $r$, Let $\Psi(x^{[n]})=\psi(e_1(x^{[n]}))\cdots\psi(e_{rn}(x^{[n]}))$, where $e_1(x^{[n]}),\cdots,e_{rn}(x^{[n]})$ are Chern roots of $x^{[n]}$. It is obvious to see that such $\Psi$ satisfies the conditions in Theorem \ref{thm:Universal}.

Let $v=(v_1=a_1 t,\cdots,v_r=a_r t)$ and $w_i=c_i t$. Using the localization formula, the equivariant version of $H_{\Psi,\Phi}$ for $\bC$ and $\mathcal{E}_r^A$ which we denote by $H_{\Psi,\Phi}(\bC,\mathcal{E}_r^A)(t)$ is as follows:
\begin{eqnarray*}
&&H_{\Psi,\Phi}(\bC,\mathcal{E}_r^A)(t)=\sum\limits_{n=0}^{\infty}z^n \frac{\phi(-st)}{-st}\prod\limits_{j=1}^{r}\psi(v_j+(s-1)t).
\end{eqnarray*}

Assume that 
\begin{eqnarray*}
H_{\Psi,\Phi}(\bC,\mathcal{E}_r^A)(t)=\exp(\sum\limits_{n=1}^{\infty}z^n\int^t_{\bC}(\sum\limits_{j=0}^{r}(A^n_{0,j}c^t_0(\bC)c^t_j(\mathcal{E}_r^A)+A^n_{1,j}c^t_1(\bC)c^t_j(\mathcal{E}_r^A)))\\
=\exp(\sum\limits_{n=1}^{\infty}z^n\int^t_{\bC}(\sum\limits_{j=0}^{r}(A^n_{0,j}\frac{\sigma_j(v_1,\cdots,v_r)}{t}+A^n_{1,j}\sigma_j(v_1,\cdots,v_r))),
\end{eqnarray*}

where we denote by $\int^t$ the equivariant integral, $c^t_i$ is the equivariant Chern classes and $\sigma_j$ is the $j$-th elementary symmetric polynomial.

Denote by $H_{\Psi,\Phi}(C,E)(t)$ the equivariant version of $H_{\Psi,\Phi}(C,E)$. It can be computed by localization as follows:
\begin{equation*}
\begin{split}
&H_{\Psi,\Phi}(C,E)(t)=\sum\limits_{n=0}^{\infty}z^n \prod\limits_{i=1}^{l}\prod\limits_{s=1}^{n_i}\frac{\phi(-sw_i)}{-sw_i}\prod\limits_{j=1}^{r}\psi(v_j+(s-1)w_i)\\
=&\prod\limits_{i=1}^{l}\sum\limits_{n_i=0}^{\infty}z^{n_i}\prod\limits_{s=1}^{n_i}\frac{\phi(-sw_i)}{-sw_i}\prod\limits_{j=1}^{r}\psi(v_j+(s-1)w_i)\\
=&\prod\limits_{i=1}^{l}H_{\Psi,\Phi}(\bC,\mathcal{E}_r^A)(w_i)\\
=&\prod\limits_{i=1}^{l}\exp(\sum\limits_{n=1}^{\infty}z^{n}\int^t_{\bC}(\sum\limits_{j=0}^{r}(A^{n}_{0,j}\frac{\sigma_j(v_1,\cdots,v_r)}{w_i}+A^{n}_{1,j}\sigma_j(v_1,\cdots,v_r)))\\
=&\exp(\sum\limits_{n=1}^{\infty}z^{n}(\sum\limits_{j=0}^{r}(A^{n}_{0,j}\sum\limits_{i=1}^{l}\frac{\sigma_j(v_1,\cdots,v_r)}{w_i}+A^{n}_{1,j}\sum\limits_{i=1}^{l}\sigma_j(v_1,\cdots,v_r)))\\
=&\exp(\sum\limits_{n=1}^{\infty}z^n\int^t_{C}(\sum\limits_{j=0}^{r}(A^n_{0,j}c^t_0(C)c^t_j(E^A)+A^{n}_{1,j}c^t_1(\bC)c^t_j(E))).
\end{split}
\end{equation*}

If $C$ is projective, by taking nonequivariant limit one has 
\begin{equation*}
H_{\Psi,\Phi}(C,E)=\exp(\sum\limits_{n=1}^{\infty}z^n\int_{C}(\sum\limits_{j=0}^{r}(A^n_{0,j}c_0(C)c_j(E^A)+A^{n}_{1,j}c_1(\bC)c_j(E)))
\end{equation*}

\subsection{Proof of Theorem \ref{thm:Cn_1}}\label{subset:proof}

Let us see the general case of Conjecture \ref{conj:Cn_n} first. Taking $\Psi:K(X)\to H^{\times}$ to be the total Chern class and $\Phi=1$, we see that such $H_{\Psi,\Phi}$ satisfies the conditions in Theorem \ref{thm:Universal} and hence can be written in the desired form. So $\ds\frac{1}{n}A_n^r$, $\ds\frac{1}{n}B_n^r$, $\ds\frac{1}{n}C_n^r$ and $\ds\frac{1}{n}D_n^r$ exist as the n-th coefficients of the corresponding universal power series. Now clearly $A_n^r$, $B_n^r$, $C_n^r$ and $D_n^r$ are rational numbers depending only on $r$ and $n$, and we hope to determine them explicitly as integers. As we have discussed in Section \ref{subsec:loc}, in order to prove Conjecture \ref{conj:Cn_n}, it suffices to prove the equivariant version of $(\bC,\mathcal{E}_r^A)$. Using localization, we have checked Conjecture \ref{conj:Cn_n} for $r=2,3,4,5$ and $n<10$. We also conjecture that $$B^2_n=(-1)^n(4^n-\ds\binom{2n-1}{n-1}),$$ $$B^3_n=(-1)^n (\ds\sum\limits_{i=0}^{n-1}(\ds\frac{2^{n-2-i}}{n}(n-i)(3n-3i-1)\ds\binom{3n}{i}-\ds\binom{3n-1}{n-1})$$ and have checked them up to $n=10$. However effort fails to find the explicit expressions for higher $r$.

Before proving Theorem \ref{thm:Cn_1}, recall that in Section 2.5 of \cite{wang2014tautological} the following identity is established:
\begin{equation} \label{eqn:Cn}
\begin{split}
&\sum_{n=0}^{\infty} z^n \chi(\bC^{[n]}, \Lambda_{-y}(\mathcal{E}^A_1)^{[n]})(q)\\
=& \sum_{n=0}^{\infty} z^n \prod_{i=1}^n \frac{1 - y q^a q^{i-1}}{1 - q^i}\\
= &\exp ( \sum_{n=1}^{\infty} \frac{z^n}{n} \frac{1-q^{na} y^n}{1-q^n})\\
=&\exp ( \sum_{n=1}^{\infty} \frac{z^n}{n} \chi(\bC, \Lambda_{-y^n}\mathcal{E}^A_1)(q^n)).
\end{split}
\end{equation}
Here $\Lambda_u E = \sum\limits_{i=0}^n u^i\Lambda^iE$ and $\chi(\bC^{[n]}, \Lambda_{-y}(\mathcal{O}_{\bC})^{[n]})(q)$ the equivariant Euler characteristic of $\Lambda_{-y}(\mathcal{O}_{\bC})^{[n]}$ on $\bC^{[n]}$.

Moreover, the discussion in the last subsection implies that (\ref{eqn:Cn}) can be generalized as the following:

\begin{proposition}
For a smooth projective curve $C$ and a line bundle $L$ on $C$,
\begin{equation*}
\sum_{n=0}^{\infty} z^n \chi(C^{[n]}, \Lambda_{-y}L^{[n]})=\sum_{n=0}^{\infty} z^n \chi(C, \Lambda_{-y^n}L).
\end{equation*}

\end{proposition}

To prove Theorem \ref{thm:Cn_1}, we only need to prove the following lemma by using (\ref{eqn:Cn}):

\begin{lemma}
\begin{eqnarray*}
&& \sum_{n=0}^{\infty} z^n \int^t_{C^{[n]}} c^t((\mathcal{E}_1^A)^{[n]}) =
\exp ( \sum_{n=1}^{\infty} (-1)^{n+1} z^n\int^t_{\bC} c^t(\mathcal{E}_1^A)),
\end{eqnarray*}
where $A=(a)$, $a\in\bZ$ is the equivariant weight of $\mathcal{O}_{\bC}$.
\end{lemma}

\begin{proof}
Similarly as in \cite{Iqu-Naz-Raz-Sal}, let $q=\exp(\beta t)$, $y=\exp(\beta)$, and take $\beta\rightarrow 0$ in (\ref{eqn:Cn}), one has
\begin{equation*}
\begin{split}
&\sum_{n=0}^{\infty} z^n \prod_{i=1}^n \frac{1+at+(i-1)t}{it}\\
= &\exp ( \sum_{n=1}^{\infty} \frac{z^n}{n}\frac{n+ant}{nt})\\
=&\exp ( \sum_{n=1}^{\infty} \frac{z^n}{n}\frac{1+at}{t}).
\end{split}
\end{equation*}

Hence
\begin{equation*}
\begin{split}
&\sum_{n=0}^{\infty} z^n \int^t_{\bC^{[n]}} c^t((\mathcal{E}^A_1)^{[n]})\\
=&\sum_{n=0}^{\infty} z^n \prod_{i=1}^n \frac{1+at+(i-1)t}{-it}\\
=&\sum_{n=0}^{\infty} (-z)^n \prod_{i=1}^n \frac{1+at+(i-1)t}{it}\\
=&\exp ( \sum_{n=1}^{\infty} \frac{(-z)^n}{n}\frac{1+at}{t})\\
=&\exp ( \sum_{n=1}^{\infty} \frac{(-1)^{n+1}}{n} z^n\frac{1+at}{-t})\\
=&\exp ( \sum_{n=1}^{\infty} \frac{(-1)^{n+1}}{n} z^n\int^t_{\bC} c^t(\mathcal{E}_1^A)).
\end{split}
\end{equation*}

\end{proof}

\subsection{Proof of Theorem \ref{thm:Seg1}}

We also have an equivariant version of (\ref{eqn:Seg1}) of $\bC$. However, it is also difficult to compute. We have to find some other way to give a proof.

Denote $\mathcal{O}(d)$ over $\bP^1$ by $L_d$ and $\Int_{(\bP^1)^{[n]}}s(L_d^{[n]})$ by $N_n^d$. As it has been discussed, Theorem \ref{thm:Seg1} is true if the following lemma holds:

\begin{lemma} \label{lem:SegP01}

\begin{eqnarray} \label{eqn:Seg0P1}
&& \sum_{n=0}^{\infty} z^n N_n^0=
\frac{(1-k)^{2}}{1-2k},
\end{eqnarray}
and
\begin{eqnarray} \label{eqn:Seg1P1}
&& \sum_{n=0}^{\infty} z^n N_n^{-1}=
\frac{1-k}{1-2k}.
\end{eqnarray}
Here $z=k(1-k)$.

\end{lemma}

We will use the localization formula to prove this lemma. Recall that the homogeneous coordinates on $\bP^1$ are given by $[\zeta_1:\zeta_2]$ and there is a torus-action on $\bP^1$: $$q \cdot [\zeta_1:\zeta_2] =[\zeta_1: q \cdot \zeta_2]$$.

There are two fixed points $P_1=[1:0]$ and $P_2=[0:1]$ on $\bP^1$. We choose the canonical lifting to the tangent bundle of $\bP^1$ and the weight decompositions of the cotangent space at $P_1$ and $P_2$ are given by $q^{-1}$ and $q$ respectively. We also choose a lifting to $L_d$ such that the weight decomposition of $L_d$ is given by 
$L_d|_{P_1}=1$ and $L_d|_{P_1}=q^{-d}$. Denote the equivariant integral $\Int_{(\bP^1)^{[n]}}s^t_x(L_d^{[n]})$ by $N_n^d(t)$, where $$s^t_x(L_d^{[n]})=\ds\frac{1}{c^t_x(L_d^{[n]})}=\ds\frac{1}{1+xc^t_1(L_d^{[n]})+\cdots+x^nc^t_n(L_d^{[n]})}$$ is the equivariant total Segre class. $N_n^d$ is the coefficient of $x^n$ in $N_n^d(t)$. By localization formula one has
\begin{equation*}
N_n^d(t)=\sum\limits_{k=0}^{n}\prod\limits_{i=1}^{k}\frac{1}{(1-x(i-1)t)(it)}\prod\limits_{i=1}^{n-k}\frac{1}{(1+x((i-1)t-dt))(-it)}.
\end{equation*}
Here we write $\prod\limits_{i=s}^{s-1}(\cdot)=1$ for convenient notation.

We have the following lemma:

\begin{lemma} \label{lem:Segd}
One has
\begin{equation*}
N_n^d(t)=\frac{\dbinom{2n-2-d}{n} x^n}{\prod\limits_{i=0}^{n-1}(1+(d-i)xt)(1-ixt)},
\end{equation*}
and hence it is easy to see by comparing the coefficient that $N_n^d=\ds\binom{2n-2-d}{n}$.
\end{lemma}

\begin{proof}
\begin{equation*}
\begin{split}
&N_n^d(t)\\
=&\sum\limits_{k=0}^{n}\prod\limits_{i=1}^{k}\frac{1}{(1-x(i-1)t)(it)}\prod\limits_{i=1}^{n-k}\frac{1}{(1+x((i-1)t-dt))(-it)}\\
\overset{y=\frac{1}{tx}}{=}&\sum\limits_{k=0}^{n}\prod\limits_{i=1}^{k}\frac{1}{(1-\frac{(i-1)}{y})(it)}\prod\limits_{i=1}^{n-k}\frac{1}{(1+\frac{(i-1)-d}{y})(-it)}\\
=&\sum\limits_{k=0}^{n}y^n \prod\limits_{i=1}^{k}\frac{1}{(y-(i-1))(it)}\prod\limits_{i=1}^{n-k}\frac{1}{(y+((i-1)-d))(-it)}\\
=&\frac{y^n}{\prod\limits_{i=0}^{n-1}(y+(d-i))(y-i)t^n} \sum\limits_{k=0}^{n} \frac{\prod\limits_{i=k}^{n-1}(-y+i)\prod\limits_{i=n-k}^{n-1}(y+i-d)}{k!(n-k)!}\\
=&\frac{y^n}{\prod\limits_{i=0}^{n-1}(y+(d-i))(y-i)t^n} \sum\limits_{k=0}^{n} \binom{-y+n-1}{n-k} \binom{y+n-1-d}{k}\\
=&\frac{y^n}{\prod\limits_{i=0}^{n-1}(y+(d-i))(y-i)t^n}\dbinom{2n-2-d}{n}.
\end{split}
\end{equation*}

The last identity is a special case of the Chu-Vandermonde's identity (cf. \cite{koepf2014hypergeometric}, P. 45 Exercise 3.2 (a)).

Take $x=\ds\frac{1}{ty}$ and one gets $$N_n^d(t)=\ds\frac{\ds\binom{2n-2-d}{n} x^n}{\prod\limits_{i=0}^{n-1}(1+(d-i)xt)(1-ixt)}.$$

\end{proof}

Now we are going to prove (\ref{eqn:Seg0P1}).

($N_n^0$) is  the integer sequence A001791 in the on-line encyclopedia of integer sequences \cite{sloane2007line} and the generating series is given by
\begin{equation*}
\sum\limits_{n=0}^{\infty} N_n^0 z^n=\frac{1-2z+\sqrt{1-4z}}{2\sqrt{1-4z}}.
\end{equation*}

Take $z=k(1-k)$ and one can easily get (\ref{eqn:Seg0P1}).

Applying similar arguments we can prove (\ref{eqn:Seg1P1}), so we omit the proof here.


\begin{thebibliography}{99}

\bibitem{arbarello2011geometry}
Arbarello E, Cornalba M, Griffiths P.
\newblock Geometry of Algebraic Curves: Volume II with a contribution by Joseph
  Daniel Harris, volume 268.
\newblock Springer Science \& Business Media, 2011


\bibitem{Ati-Sin}
Atiyah M~F, Singer I~M.
\newblock The index of elliptic operators: III.
\newblock Annals of mathematics,  546--604 (1968)


\bibitem{beltrametti1991zero}
Beltrametti M, Sommese A~J.
\newblock Zero cycles and k-th order embeddings of smooth projective surfaces.
\newblock Proceedings of Problems in the theory of surfaces and their
  classification, Symposia Math, volume~32, 1991.
\newblock  33--48


\bibitem{beltrametti1990k}
Beltrametti M, Sommese A~J.
\newblock On k-spannedness for projective surfaces.
\newblock Springer (1990)


\bibitem{Bor-Lib}
Borisov L, Libgober A.
\newblock McKay correspondence for elliptic genera.
\newblock Annals of mathematics,  1521--1569 (2005)


\bibitem{catanese1990d}
Catanese F, G{\oe}ttsche L.
\newblock d-very-ample line bundles and embeddings of {H}ilbert schemes of
  0-cycles.
\newblock manuscripta mathematica, 68(1):337--341 (1990)


\bibitem{cotterill2011geometry}
Cotterill E.
\newblock Geometry of curves with exceptional secant planes: linear series
  along the general curve.
\newblock Mathematische Zeitschrift, 267(3-4):549--582 (2011)


\bibitem{Ell-Got-Leh}
Ellingsrud G, G{\"o}ttsche L, Lehn M.
\newblock On the cobordism class of the {H}ilbert scheme of a surface.
\newblock Journal of Algebraic Geometry, 10:81--100 (2001)


\bibitem{ellingsrud1996bott}
Ellingsrud G, Str${\o}$mme S.
\newblock Bott’s formula and enumerative geometry.
\newblock Journal of the American Mathematical Society, 9(1):175--193 (1996)


\bibitem{Got}
G{\"o}ttsche L.
\newblock The {B}etti numbers of the {H}ilbert scheme of points on a smooth
  projective surface.
\newblock Mathematische Annalen, 286(1):193--207 (1990)


\bibitem{Got-Soe}
G{\"o}ttsche L, Soergel W.
\newblock Perverse sheaves and the cohomology of {H}ilbert schemes of smooth
  algebraic surfaces.
\newblock Mathematische Annalen, 296(1):235--245 (1993)


\bibitem{Iqu-Naz-Raz-Sal}
Iqbal A, Nazir S, Raza Z, et~al.
\newblock Generalizations of Nekrasov-Okounkov Identity.
\newblock Annals of Combinatorics, 16(4):745--753 (2012)


\bibitem{koepf2014hypergeometric}
Koepf W.
\newblock Hypergeometric summation.
\newblock Braunschweig/Wiesbaden: Vieweg, 2014


\bibitem{le1987formules}
Le~Barz P.
\newblock Formules pour les multisecantes des surfaces algebriques.
\newblock L'Ens. Math, 33:1--66 (1987)


\bibitem{le2007formule}
Le~Barz P.
\newblock Sur une formule de {C}astelnuovo pour les espaces multis{\'e}cants.
\newblock Bollettino dell unione matematica italiana. Sezione B: articoli di
  ricerca matematica, 10(2):381--388 (2007)


\bibitem{Leh}
Lehn M.
\newblock Chern classes of tautological sheaves on {H}ilbert schemes of points
  on surfaces.
\newblock Inventiones mathematicae, 136(1):157--207 (1999)


\bibitem{Mar-Opr}
Marian A, Oprea D.
\newblock Tautological integrals over symmetric powers of curves.\\
\newblock http://math.ucsd.edu/~doprea/segre-curves.pdf (2015)


\bibitem{Mar-Opr-Pan}
Marian A, Oprea D, Pandharipande R.
\newblock Segre classes and Hilbert schemes of points.
\newblock arXiv preprint arXiv:1507.00688 (2015)


\bibitem{Nak2}
Nakajima H.
\newblock Lectures on {H}ilbert schemes of points on surfaces.
\newblock Number~18, American Mathematical Soc., 1999


\bibitem{rennemo2012universal}
Rennemo J~V.
\newblock Universal polynomials for tautological integrals on {H}ilbert
  schemes.
\newblock arXiv preprint arXiv:1205.1851 (2012)


\bibitem{sloane2007line}
Sloane N~J.
\newblock The on-line encyclopedia of integer sequences.
\newblock Pubilished electronically at https://oeis.org


\bibitem{tikhomirov1994standard}
Tikhomirov A.
\newblock Standard bundles on a {H}ilbert scheme of points on a surface.
\newblock Proceedings of Algebraic Geometry and its Applications. Vieweg+ Teubner Verlag, 1994: 183-203


\bibitem{tikhomirov1994top}
Tikhomirov A, Troshina T.
\newblock Top {S}egre Class of a Standard Vector Bundle $\varepsilon_D^4$ on
  the {H}ilbert Scheme $Hilb^4S$ of a Surface ${S}$.
\newblock Proceedings of Algebraic Geometry and its Applications. Vieweg+ Teubner Verlag, 1994: 205--226


\bibitem{troshina1994degree}
Troshina T.
\newblock The degree of the top {S}egre class of the standard vector bundle on
  the {H}ilbert scheme ${H}ilb^4S$ of an algebraic surface ${S}$.
\newblock Izvestiya Rossiiskoi Akademii Nauk. Seriya Matematicheskaya, 1993, 57(6): 106-129

\bibitem{wang2014tautological}
Wang Z, Zhou J.
\newblock Tautological sheaves on {H}ilbert schemes of points.
\newblock Journal of Algebraic Geometry, 23(4):669--692 (2014)

\end{thebibliography}
\end{document}